\documentclass[a4paper,twopage,reqno,12pt]{amsart}
\usepackage[top=30mm,right=30mm,bottom=30mm,left=30mm]{geometry}

\usepackage{pdfsync}
\usepackage{graphicx}
\usepackage{amsmath}
\usepackage{amssymb}
\usepackage{amsfonts}
\usepackage{amsthm}
\usepackage{amstext}
\usepackage{amsbsy}
\usepackage{amsopn}
\usepackage{amscd}
\usepackage{enumerate}
\usepackage{color}
\usepackage[colorlinks]{hyperref}
\usepackage[hyperpageref]{backref}
\usepackage{multirow}
\usepackage{lipsum}
\usepackage{rotating}
\usepackage{longtable}
\usepackage{float}
\usepackage{lscape}
\usepackage{pdflscape}
\usepackage[]{algorithm2e}

\newtheorem{theorem}{Theorem}[section]

\newtheorem{lemma}[theorem]{Lemma}

\theoremstyle{definition}

\numberwithin{equation}{section}

\newcommand{\GL}{\mathrm{GL}}
\newcommand{\GU}{\mathrm{GU}}

\newcommand{\SO}{\mathrm{SO}}
\newcommand{\SU}{\mathrm{SU}}

\newcommand{\PSL}{\mathrm{PSL}}
\newcommand{\PSU}{\mathrm{PSU}}
\newcommand{\PSp}{\mathrm{PSp}}

\newcommand{\PGaL}{\mathrm{P\Gamma L}}

\newcommand{\POm}{\mathrm{P \Omega}}
\newcommand{\Sp}{\mathrm{Sp}}

\newcommand{\D}{\mathrm{D}}

\newcommand{\Aut}{\mathrm{Aut}}
\newcommand{\Out}{\mathrm{Out}}
\newcommand{\Inn}{\mathrm{Inn}}

\renewcommand{\S}{\mathrm{S}}
\newcommand{\A}{\mathrm{A}}
\newcommand{\PG}{\mathrm{PG}}

\newcommand{\F}{\mathbb{F}}

\newcommand{\Dmc}{\mathcal{D}}
\newcommand{\Pmc}{\mathcal{P}}
\newcommand{\Bmc}{\mathcal{B}}

\renewcommand{\leq}{\leqslant}
\renewcommand{\geq}{\geqslant}

\renewcommand{\mod}[1]{\ (\mathrm{mod}{\ #1})}
\newcommand{\imod}[1]{\allowbreak\mkern4mu({\operator@font mod}\,\,#1)}

\begin{document}

 \title[]{Symmetric designs and projective special unitary groups of dimension at most five}

\author[A. Daneshkhah]{Asharf Daneshkhah}%
 \address{Asharf Daneshkhah, Department of Mathematics, Faculty of Science, Bu-Ali Sina University, Hamedan, Iran.
 }%
  \email{adanesh@basu.ac.ir}

 \subjclass[]{05B05; 05B25; 20B25}%
 \keywords{Symmetric design; flag-transitive; point-primitive; automorphism group}
 \date{\today}%

\begin{abstract}
In this article, we study symmetric $(v, k, \lambda)$ designs admitting a flag-transitive and point-primitive automorphism group $G$ whose socle is a projective special unitary group of dimension at most five. We, in particular, determine all such possible parameters $(v, k, \lambda)$ and show that there exist eight non-isomorphic of such designs for which $\lambda\in\{3,6,12, 16, 18\}$ and $G$ is $\PSU_{3}(3)$, $\PSU_{3}(3):2$, $\PSU_{4}(2)$ or $\PSU_{4}(2):2$.
\end{abstract}

\maketitle

\section{Introduction}\label{sec:intro}

A \emph{symmetric $(v,k,\lambda)$ design} is an incidence structure $\Dmc=(\Pmc,\Bmc)$  consisting of a set $\Pmc$ of $v$ \emph{points} and a set $\Bmc$ of $v$ \emph{blocks} such that every point is incident with exactly $k$ blocks, and every pair of blocks is incident with exactly $\lambda$ points. A \emph{nontrivial} symmetric design is one in which $2<k<v-1$.
A \emph{flag} of $\Dmc$ is an incident pair $(\alpha,B)$ where $\alpha$ and $B$ are a point and a block of $\Dmc$, respectively. An \emph{automorphism} of a symmetric design $\Dmc$ is a permutation of the points permuting the blocks and preserving the incidence relation. An automorphism group $G$ of $\Dmc$ is called \emph{flag-transitive} if it is transitive on the set of flags of $\Dmc$. If $G$ is primitive on the point set $\Pmc$, then $G$ is said to be \emph{point-primitive}. 
We here adopt the standard notation for finite simple groups of Lie type, for example, we use $\PSL_{n}(q)$, $\PSp_{n}(q)$, $\PSU_{n}(q)$, $\POm_{2n+1}(q)$ and $\POm_{2n}^{\pm}(q)$ to denote the finite classical simple groups. Symmetric and alternating groups on $n$ letters are denoted by $S_{n}$ and $A_{n}$, respectively.  We denote by $n$ the cyclic group of order $n$, and we write $E_{n}$ for an elementary abelian group of order $n$. A group $G$ is said to be \emph{almost simple} with socle $X$ if $X\unlhd G\leq \Aut(X)$, where $X$ is a nonabelian simple group. Further notation and definitions in both design theory and group theory are standard and can be found, for example in \cite{b:beth1999design,b:Atlas,b:Dixon,b:Hugh-design}. 

The main aim of this paper is to study flag-transitive symmetric designs. In \cite{a:Praeger-imprimitive}, Praeger and Zhou study point-imprimitive symmetric $(v,k,\lambda)$ designs and give a classification of such designs in terms of their parameters. In the case where, a symmetric design admits a flag-transitive and point-primitive automorphism group, for small $\lambda$ (i.e. $\lambda\leq 100$), the only type of primitive groups might occur is almost simple or affine \cite{a:reg-reduction,a:Zhou-lam100}. Although, it is still unknown for larger $\lambda$ such an automorphism group is of these two types, it is somehow interesting to study such designs whose automorphism group $G$ is an almost simple group with socle $X$. This paper is part of contribution in classification of symmetric designs admitting flag-transitive and point-primitive finite almost simple automorphism groups of Lie type of small dimension, see \cite{a:ABD-PSL2, a:ABD-PSL3,a:ABD-Exp,a:ABDZ-U4,a:D-PSU3}. In this paper, we continue this project to obtain symmetric designs admitting an automorphism group whose socle is a projective special unitary group of dimension at most five.

\begin{theorem}\label{thm:main}
 Let $\Dmc$ be a nontrivial symmetric $(v, k, \lambda)$ design and let $\alpha$ be a point of $\Dmc$. If $G$ is a flag-transitive and point-primitive automorphism group of $\Dmc$ of almost simple type with socle $X$ is a projective special unitary groups of dimension at most five over a finite field of size $q$, then $\lambda\in\{3,6,12, 16, 18\}$ and $v$, $k$, $\lambda$, $X$, $G_{\alpha}\cap X$ and $G$ are as in one of the lines in {\rm Table~\ref{tbl:main}}.
\end{theorem}
\begin{table}
    \centering
    \scriptsize
    \caption{Parameters in Theorem~\ref{thm:main}}\label{tbl:main}
    \begin{tabular}{p{.4cm}llllllp{3.5cm}l}
        \hline\noalign{\smallskip}
        Line &
        \multicolumn{1}{c}{$v$} &
        \multicolumn{1}{c}{$k$} &
        \multicolumn{1}{c}{$\lambda$} &
        \multicolumn{1}{c}{$X$} &
        \multicolumn{1}{c}{$H$}  &
        \multicolumn{1}{c}{$G$}  &
        \multicolumn{1}{c}{Designs} &
        \multicolumn{1}{c}{References$^{\ast}$} \\
        \noalign{\smallskip}\hline\noalign{\smallskip}
        $1$ & $36$ & $21$ & $12$ & $\PSU_{3}(3)$  & $\PSL_2(7)$  & $\PSU_{3}(3)$
        & Menon
        &\cite{a:Braic-2500-nopower} \\
        $2$ & $36$ & $21$ & $12$ & $\PSU_{3}(3)$  & $\PSL_2(7):2$  & $\PSU_{3}(3):2$
        & Menon
        &\cite{a:Braic-2500-nopower} \\
        $3$ & $36$ & $15$ & $6$ & $\PSU_{4}(2)$  & $S_{6}$  & $\PSU_{4}(2)$
        & Menon
        &\cite{a:ABDZ-U4} \\
        $4$ & $36$ & $15$ & $6$ & $\PSU_{4}(2)$  & $\S_{6}:2$  & $\PSU_{4}(2):2$
        & Menon
        &\cite{a:ABDZ-U4} \\
        $5$ & $40$ & $27$ & $18$ & $\PSU_{4}(2)$  & $3_{+}^{1+2}:2A_{4}$  & $\PSU_{4}(2)$
        &Complement of $\PG_{3}(3)$
        &\cite{a:ABDZ-U4} \\
        $6$ & $40$ & $27$ & $18$ & $\PSU_{4}(2)$  & $3_{+}^{1+2}:2A_{4}:2$  & $\PSU_{4}(2):2$
        &Complement of $\PG_{3}(3)$
        &\cite{a:ABDZ-U4} \\
        $7$ & $40$ & $27$ & $18$ & $\PSU_{4}(2)$  & $3^3:S_{4}$  & $\PSU_{4}(2)$
        & Complement of Higman design
        &\cite{a:ABDZ-U4} \\
        $8$ & $40$ & $27$ & $18$ & $\PSU_{4}(2)$  & $3^3:S_{4}:2$  & $\PSU_{4}(2):2$
        & Complement of Higman design
        &\cite{a:ABDZ-U4} \\
        $9$ & $45$ & $12$ & $3$ & $\PSU_{4}(2)$  & $2^{.}(A_{4}\times A_{4}).2$ & $\PSU_{4}(2)$
        &-
        &\cite{a:ABDZ-U4, a:Praeger-45-12-3}\\
        $10$ & $45$ & $12$ & $3$ & $\PSU_{4}(2)$  & $2^{.}(A_{4}\times A_{4}).2:2$ & $\PSU_{4}(2):2$
        &-
        &\cite{a:ABDZ-U4, a:Praeger-45-12-3}\\
        $11$ & $63$ & $32$ & $16$ & $\PSU_{3}(3)$  & $4.\S_4$  & $\PSU_{3}(3)$
        &-
        &\cite{a:Braic-2500-nopower} \\
        $12$ & $63$ & $32$ & $16$ & $\PSU_{3}(3)$  & $2_{+}^{1+4}.\S_3$  & $\PSU_{3}(3):2$
        &-
        &\cite{a:Braic-2500-nopower} \\
        $13$ & $63$ & $32$ & $16$ & $\PSU_{3}(3)$  & $4^2:\S_3$  & $\PSU_{3}(3)$
        &-
        &\cite{a:Braic-2500-nopower} \\
        $14$ & $63$ & $32$ & $16$ & $\PSU_{3}(3)$  & $4^2:\D_{12}$  & $\PSU_{3}(3):2$
        &-
        &\cite{a:Braic-2500-nopower} \\
        \noalign{\smallskip}\hline
        \multicolumn{9}{p{12cm}}{\tiny $\ast$ The last column addresses to references in which a design with the parameters in the line has been constructed.}
    \end{tabular}
\end{table}

In Section~\ref{sec:examples}, we give  detailed information about the designs obtained in Theorem~\ref{thm:main} in particular those appear in Table \ref{tbl:main}. We also use the software \textsf{GAP} \cite{GAP4} for computational arguments.

\section{Examples and comments}\label{sec:examples}

In this section, we provide some examples of symmetric designs admitting flag-transitive and point-primitive automorphism groups. We, in particular, make some  comments on Theorem~\ref{thm:main} and the designs mentioned in Table~\ref{tbl:main}. In what follows, suppose that $V$ is a vector space of dimension $n$ over a finite field $\F_{q}$ of size $q$.

A well-known example of flag-transitive symmetric designs is a projective space $\PG(n-1,q)$ which is an incidence structure whose points are $1$-dimensional subspaces of $V$ and the lines are $(n-1)$-dimensional of $V$, and the incidence is given by inclusion. The full automorphism group of this design is the projective semilinear group $\PGaL(n,q)$ with socle $\PSL(n, q)$. The unique symmetric $(7, 3, 1)$ design known also as \emph{Fano Plane} is in indeed $\PG(2, 2)$ admitting flag-transitive and point-primitive automorphism group $\PSL(3,2)\cong \PSL(2,7)$. The symmetric $(7,4,2)$ design is the complement of Fano plane which is also flag-transitive and point-primitive , see \cite[Section 1.2.1]{a:reg-reduction}. These examples appear in Theorem~\ref{thm:main}(a)-(b) and line 1 of  Table~\ref{tbl:main}.

The design in lines 2 of Table \ref{tbl:main} is the unique symmetric $(11,5, 2)$ design as a \emph{Paley difference set} which is in fact a Hadamard design, and its full automorphism group is $\PSL(2,11)$ acting  flag-transitively and point-primitively. In this case, the point-stabiliser is isomorphic to $\A_{5}$, see \cite[Section 1.2.1]{a:reg-reduction}. The complement of this design, which is the one in line 3 of Table~\ref{tbl:main}, is the unique symmetric $(11,6,3)$ design whose full automorphism group $\PSL(2,11)$ is also flag-transitive and point-primitive with $\A_{5}$ as its point-stabiliser, see also \cite[Theorem 1.2]{a:Zhou-lam3-classical}.

The unique symmetric $(15,8,4)$ design $\Dmc$ in line 4 of Table~\ref{tbl:main} can be  constructed by points and complements of hyperplanes of $\PG(3,2)$. The full automorphism group of $\Dmc$ is $\PSL(4,2)\cong \A_{8}$ which admits a proper subgroup $\PSL(2,9)\cong \A_{6}$ as an automorphism group of $\Dmc$. Note that $\PSL(2,9)$ acts flag-transitively on $\Dmc$ with  point-stabiliser $\S_{4}$, but not on its complement $\Dmc^{\ast}$. In fact, it has two orbits on the set of flags of $\Dmc^{\ast}$ of size $15$ and $15\cdot 6$, respectively, namely, the stabiliser of a point $\alpha$ has two orbits of size $1$ and $6$ respectively on the seven planes containing $\alpha$, see \cite[Theorem 1.2]{a:Zhou-lam3-classical}.

The symmetric designs with parameters $(35, 18, 9)$ in lines 5-6 in Table~\ref{tbl:main}  are the complement of the symmetric $(35, 17, 8)$ designs which are Hadamard designs with parameters $(4n-1, 2n-1, n-1)$ for $n=9$. These designs arise from studying rank three permutation groups, see~\cite{a:Braic-2500-nopower,b:Atlas,a:rank3}.

The designs in lines 7-11 in Table~\ref{tbl:main} also arise from studying rank three permutation groups. The complement of the symmetric $(56,11,2)$ design is the design with parameters $(56,45,18)$ admitting a rank three automorphism group $\PSL(3, 4)$ as its socle, see~\cite{a:Braic-2500-nopower} and \cite[Theorem]{a:rank3}.

\section{Preliminaries}\label{sec:pre}

In this section, we state some useful facts in both design theory and group theory. Recall that a group $G$ is called almost simple if $X\unlhd G\leq \Aut(X)$, where $X$ is a nonabelian simple group. If $H$ is a maximal subgroup not containing the socle $X$ of an almost simple group $G$, then $G=HX$, and since we may identify $X$ with $\Inn(X)$, the group of inner automorphisms of $X$, we also conclude that $|H|$ divides $|\Out(X)|\cdot |X\cap H|$. This implies the following elementary and useful fact:

\begin{lemma}\label{lem:New}{\rm \cite[Lemma 2.2]{a:ABD-PSL3}}
Let $G$  be an almost simple group with socle $X$, and let $H$ be maximal in $G$ not containing $X$. Then $G=HX$ and $|H|$ divides $|\Out(X)|{\cdot}|X\cap H|$.
\end{lemma}

\begin{lemma}\label{lem:Tits}
Suppose that $\Dmc$ is a symmetric $(v,k,\lambda)$ design admitting a flag-transitive and point-primitive almost simple automorphism group $G$ with socle $X$ of Lie type in odd characteristic $p$. Suppose also that the point-stabiliser $G_{\alpha}$, not containing $X$, is not a parabolic subgroup of $G$. Then $\gcd(p,v-1)=1$.
\end{lemma}
\begin{proof}
Note that $G_{\alpha}$ is maximal in $G$, then by Tits' Lemma \cite[1.6]{a:tits}, $p$ divides $|G:G_{\alpha}|=v$, and so  $\gcd(p,v-1)=1$.
\end{proof}

\begin{lemma}\label{lem:subdeg}{\rm \cite[3.9]{a:LSS1987}}
If $X$ is a group of Lie type in characteristic $p$, acting on the set
of cosets of a maximal parabolic subgroup, and $X$ is not $\PSL_n(q)$, $\POm_{2m}^{+}(q)$
(with $m$ odd) and $E_{6}(q)$, then there is a unique subdegree which is a power of $p$.
\end{lemma}

\begin{lemma}\label{lem:subdeg}
    Let $G$ be an almost simple group with socle $X=\PSU_{5}(q)$, and let $H$ be a maximal subgroup of $G$ with $H\cap X$ being as in the second column of \rm{Table}~\ref{tbl:subdeg}. Then the action
    of $G$ on the cosets of $H$ has subdegrees dividing the numbers $d$ listed in the last column of \rm{Table}~\ref{tbl:subdeg}.
\end{lemma}
\begin{table}
    \small
    \centering
    \caption{Some subdegrees of almost simple groups with socle $\PSU_{5}(q)$.}\label{tbl:subdeg}
    \begin{tabular}{cp{4.5cm}l}
        \hline
        Line &
        \multicolumn{1}{l}{$H\cap X$} &
        \multicolumn{1}{l}{$d$} \\
        \hline
        $1$ & $^{\hat{}}\GU_{4}(q)$  &  $(q+1)(q^{4}-1)$ \\
        $2$ & $^{\hat{}}(\SU_{3}(q)\times \SU_{2}(q)):(q+1)$ &  $(q^2-1)(q^{3}+1)$ \\
        \hline
    \end{tabular}
\end{table}
\begin{proof}
Suppose first  that $H\cap X$ is isomorphic to $^{\hat{}}\GU_{4}(q)$. In this case, $H$ stabilises a pair of non-degenerate subspaces which are mutually orthogonal and span the underlying space $V$. Thus $H=N_{G}(W)$ where $W$ is a $1$-dimensional non-degenerate subspace. Taking $\alpha=\langle u_{1}\rangle$ and $\beta=\langle u_{2}, u_{2} \rangle$, we see that $|G_{\alpha}, G_{\alpha \beta}|$ divides $(q+1)(q^{4}-1)$.
    Suppose now that $H\cap X$ is isomorphic to $\,^{\hat{}}(\SU_{3}(q)\times \SU_{2}(q)):(q+1)$. Again here $H$ stabilises a pair of non-degenerate subspaces, and so $H=N_{G}(W)$ where $W$ is a $2$-dimensional non-degenerate subspace. Set $\alpha=\langle u_{1}, u_2\rangle$ and $\beta=\langle u_{1}, u_{3} \rangle$. Then $|G_{\alpha}, G_{\alpha \beta}|$ divides $(q^2-1)(q^{3}+1)$.
\end{proof}

\begin{lemma}\label{lem:divisible}
Suppose that $\Dmc$ is a a symmetric $(v,k,\lambda)$ design. Let $G$ be a  flag-transitive automorphism group of $\Dmc$ with simple socle $X$ of Lie type in characteristic $p$. If the point-stabiliser $H=G_{\alpha}$ contains a normal quasi-simple subgroup $N$ of Lie type in characteristic $p$ and $p$ does not divide $|Z(K)|$, then $k$ is divisible by $|N{:}M|$, for some maximal subgroup $M$ of $H$.
\end{lemma}
\begin{proof}
If $B$ is a block incident with a point $\alpha$ of $\Dmc$, then $k= |H{:}H_{B}|$, and so $|N{:}N_{B}|$ divides $k$. Now, let $M$ be a maximal subgroup of $N$ such that $N_{B}\leq M$. Then $|N{:}M|$ must divide $k$, so $k$ is divisible by $|N{:}M|$.
\end{proof}

\begin{lemma}\label{lem:six}{\rm \cite[Lemma 2.1]{a:ABD-PSL2}}
Let $\Dmc$ be a symmetric $(v,k,\lambda)$ design, and let $G$ be a flag-transitive automorphism group of $\Dmc$. If $\alpha$ is a point in $\Pmc$ and $H:=G_{\alpha}$, then
\begin{enumerate}[\rm \quad (a)]
  \item $k(k-1)=\lambda(v-1)$;
  \item $4\lambda(v-1)+1$ is square;
  \item $k\mid |H|$ and $\lambda v<k^2$;
  \item $k\mid \gcd(\lambda(v-1),|H|)$;
  \item $k\mid \lambda d$, for all subdegrees $d$ of $G$.
\end{enumerate}
\end{lemma}

If a group $G$ acts primitively on a set $\Pmc$ and $\alpha\in \Pmc$ (with $|\Pmc|\geq 2$), then the point-stabiliser $G_{\alpha}$ is maximal in $G$ \cite[Corollary 1.5A ]{b:Dixon}. Therefore, in our study, we need a list of all maximal subgroups of almost simple group $G$ with socle $X:=\PSU_{5}(q)$. Note that if $H$ is a maximal subgroup of $G$, then $H_{0}:=H\cap X$ is not necessarily maximal in $X$ in which case $H$ is called a \emph{novelty}. By~\cite[Tables 8.20 and 8.21]{b:BHR-Max-Low}, the complete list of maximal subgroups of an almost simple group $G$ with socle $\PSU_{5}(q)$ are known, and in this case, there arise only three novelties.

\begin{lemma}\label{lem:maxes}
Let $G$ be an almost simple group with socle $X=\PSU_{5}(q)$, and let $H$ be a maximal subgroup of $G$ not containing $X$. Then $H_{0}:=H\cap X$ is isomorphic to one of the subgroups listed in {\rm Table~\ref{tbl:maxes}}.
\end{lemma}
\begin{proof}
The maximal subgroups $H$ of $G$ can be read off from~\cite[Tables 8.20 and 8.21]{b:BHR-Max-Low}.
\end{proof}
\begin{table}[h]
  \small
  \centering
  \caption{Maxiamal subgroups $H$ of almost simple groups with socle $X=\PSU_{5}(q)$.}\label{tbl:maxes}
\begin{tabular}{cll}
  \noalign{\smallskip}\hline\noalign{\smallskip}
   Line& $H\cap X$ & Comments  \\
  \hline\noalign{\smallskip}
  $1$ & $^{\hat{}}[q]^{1+6}:\SU_{3}(q):(q^2-1)$ &  \\
  $2$ & $^{\hat{}}[q]^{4+4}:\GL_{2}(q^2)$ &  \\
  $3$ & $^{\hat{}}\GU_{4}(q)$ & \\
  $4$ & $^{\hat{}}(\SU_{3}(q)\times \SU_{2}(q)):(q+1)$ &\\
  $5$ & $^{\hat{}}(q+1)^{4}:\S_{5}$ &\\
  $6$ & $^{\hat{}}\frac{q^5+1}{q+1}:5$ & $q\geq 3$\\
  $7$ & $^{\hat{}}\SU_{5}(q_{0})\cdot \gcd(\frac{q+1}{q_{0}+1}, 5)$ & $q=q_{0}^r$, $r$ odd prime  \\
  $8$ & $\SO_{5}(q)$ & $q$ odd \\
  $9$ & $^{\hat{}} 5_{+}^{1+2}:\Sp_{2}(5)$ & $q=p\equiv 4 \mod{5}$ or $q=p^2$ and $p\equiv 2, 3 \mod{5}$ \\
  $10$ & $\PSL_{2}(11)$ & $q=p\equiv 2,6,7,8,10 \mod{11}$ \\
  $11$ & $\PSU_{4}(2)$ & $q=p\equiv 5 \mod{6}$ \\
  \hline\noalign{\smallskip}
\end{tabular}

\end{table}

\section{\bf Proof of the main result}\label{sec:proof}

In this section, we prove Theorem~\ref{thm:main} in a series of lemmas. We first recall from  Subsection~\ref{sec:out} that the assertion for the case where $X=\PSU_{n}(q)$ with $n=3, 4$ can be deduced from \cite{a:D-PSU3,a:ABDZ-U4}. By revisiting \cite[Lemmas ??]{a:D-PSU3}, we obtain the missing designs on lines 1-2 and 11-14 of Table~\ref{tbl:main} with the same parameters.

In what follows, we suppose that $\Dmc$ is a nontrivial symmetric $(v, k, \lambda)$ design and $G$ is an almost simple automorphism group with simple socle $X=\PSU_{5}(q)$, where  $q = p^a$ ($p$ prime), that is to say, $X\lhd G \leq \Aut(X)$. Suppose also that $V=\F_{q}^{5}$ is the underlying vector space of $X$ over the finite field $\F_{q}$ of size $q$. If $G$ is a  point-primitive automorphism group of $\Dmc$, then the point-stabiliser $H = G_{\alpha}$ is maximal in $G$ \cite[Corollary 1.5A]{b:Dixon}. Let $H_{0}= H\cap X$. Then by Lemma~\ref{lem:maxes}, the subgroup $H_{0}$ is isomorphic to one of the subgroups recorded in Table~\ref{tbl:maxes}, and  so Lemma~\ref{lem:New} implies that
\begin{align}
   v=\frac{|X|}{|H_{0}|}=\frac{q^{10}(q^5+1)(q^4-1)(q^3+1)(q^2-1)}{\gcd(5, q+1){\cdot}|H_{0}|}.\label{eq:v}
\end{align}
Note that $|\Out(X)|=2a\cdot \gcd(5,q+1)$. Therefore, by Lemmas~\ref{lem:New}(b) and~\ref{lem:six}(c),
\begin{align}\label{eq:k-out}
  k \mid 2a\cdot \gcd(5, q+1){\cdot}|H_{0}|.
\end{align}

We now run through all possible subgroups $H_{0}$ recorded in Table~\ref{tbl:maxes}, and obtain the only possible cases mentioned in Theorem~\ref{thm:main}.


\begin{lemma}\label{lem:1}
The subgroup $H_{0}$ cannot be isomorphic to $\,^{\hat{}}E_{q}^{1+6}:\SU_{3}(q):(q^{2}-1)$.
\end{lemma}
\begin{proof}

By~\eqref{eq:v}, we have that $v=q^7+q^5+q^2+1$.  It follows from Lemmas~\ref{lem:six}(e) and~\ref{lem:subdeg} that $k$ divides $\lambda q^2$. Let now $m$ be a positive integer such that $mk = \lambda q^2$. Since $\lambda<k$, we have that $m<q^2$. By Lemma~\ref{lem:six}(a), $k(k-1)=\lambda(v-1)$, and so $\lambda q^2(k-1) = m\lambda(q^7+q^5+q^2)$. Thus,
\begin{align}
  k= m(q^5+q^3+1)+1 \quad \text{and} \quad
\lambda =m^2(q^3+q)+\frac{m^2+m}{q^2}.\label{eq:case1-lam}
\end{align}
Since $\lambda$ is integer, (\ref{eq:case1-lam}) implies that $q^2\mid m^2+m$. Recall that $m<q^2$. Therefore, $q^2$ must divide $m+1$, and so $m=q^2-1$. It follows from~\eqref{eq:case1-lam} that $k=(q^2-1)(q^5+q^3+1)+1=q^2(q^5-q+1)$. By~\eqref{eq:k-out}, $k$ divides $2aq^{10}(q^3+1)(q^2-1)^2$. Therefore $q^5-q+1$ must divide $2a(q^3+1)$. Thus $q^5-q+1\leq 2a(q^3+1)$, which is impossible.
\end{proof}
\begin{lemma}\label{lem:2}
The subgroup $H_{0}$ cannot be isomorphic to $^{\hat{}}[q]^{4+4}:\GL_{2}(q^2)$.
\end{lemma}
\begin{proof}
According to ~\eqref{eq:v}, we have that $v=q^8+q^5+q^3+1$. By Lemmas~\ref{lem:six}(e) and~\ref{lem:subdeg}, $k$ divides $\lambda q^3$. If $m$ is a positive integer such that $mk = \lambda q^3$, then since $\lambda<k$, we have that $m<q^3$. By Lemma~\ref{lem:six}(a), $k(k-1)=\lambda(v-1)$, and so $\lambda q(k-1)=m\lambda(q^5+q^2+1)$. Thus,
\begin{align}
  k= m(q^5+q^2+1)+1 \quad \text{and} \quad
\lambda = m^2q^2+\frac{m^2(q^2+1)+m}{q^3}.\label{eq:case2-lam}
\end{align}
Since $\lambda$ is integer, (\ref{eq:case2-lam}) implies that $q^3$ divides $m^2(q^2+1)+m$. Recall that $m<q^3$. Therefore, $q^3$ must divide $m(q^2+1)+1$. Let $n$ be a positive integer such that $m(q^2+1)+1=nq^{3}$. Note that $m<q^3$. Thus $nq^{3}=m(q^2+1)+1<q^3(q^2+1)+1$, and so $n\leq q^2+1$. Also, we have that
\begin{align*}
  m=\frac{nq^3-1}{q^2+1}=nq-\frac{nq+1}{q^2+1}.
\end{align*}
Since $m$ is integer, $q^2+1$ must divide $nq+1$. Let $s$ be a positive integer that $nq+1=s(q^2+1)$. Note that $n\leq q^2+1$. Therefore $s(q^2+1)=nq+1\leq q(q^2+1)+1$, and so $s\leq q$. As $nq+1=s(q^2+1)$, $q$  must divide $s-1$, where $s\leq q$, which is impossible.
\end{proof}

\begin{lemma}\label{lem:3}
The subgroup $H_{0}$ cannot be isomorphic to $^{\hat{}}\GU_{4}(q)$ .
\end{lemma}
\begin{proof}
We note by \eqref{eq:v} that $v=q^{4}(q^{4}-q^3+q^2-q+1)$. By Lemmas \ref{lem:subdeg} and \ref{lem:six}(e), $k$ divides $\lambda(v-1)=\lambda (q^5+q+1)(q^{2}+1)(q-1)$. Therefore, $k$ divides $\lambda(q^{2}+1)(q-1)$. Let $m$ be a positive integer that $mk=\lambda f(q)$, where $f(q)=(q^{2}+1)(q-1)$. Then by Lemma~\ref{lem:six}(a), we have that
\begin{align}
  k= 1+m(q^5+q+1)  \quad \text{and} \quad
\lambda = m^2(q^2+q)+\frac{m^2(2q+1)+m}{(q^{2}+1)(q-1)}.\label{eq:GU3-k}
\end{align}
where $m<(q^{2}+1)(q-1)$.  By Lemma \ref{lem:divisible} applied to $\PSU_{4}(q)$ we see that $k$ is divisible by the index of a maximal subgroup of $\PSU_{4}(q)$.  It follows from \cite{a:ABDZ-U4}, we see that $k$ are divisible by $q^3+1$ or $q^3$. If $q^3$ would divide $k$, then by~\eqref{eq:GU3-k}, $q^3$ should divide $m(q+1)+1$. Let $n_{1}$ be a positive integer such that $m(q+1)+1= n_{1}q^3$. Then
\begin{align*}
  m=\frac{n_{1}q^3-1}{q+1}=n_{1}(q^2-q+1)-\frac{n_{1}+1}{q+1}.
\end{align*}
Since $m$ be a positive integer, $q+1$ would divide $n_{1}+1$, and so $n_{1}>q-1$. Recall that $m<(q^{2}+1)(q-1)$. Then $n_{1}q^3=m(q+1)+1<(q^{2}+1)(q-1)(q+1)+1$, and so $n_{1}<q$, which is a contradiction. If $q^3+1$ divides $k$, then by~\eqref{eq:GU3-k}, $q^3+1$ must divide $m(q^2-q-1)+1$. If $q=2$, then $9$ must divide $m+1$, where $m<5$, which is impossible. Let now $n_{2}$ be a positive integer such that $m(q+1)+1= n_{2}(q^3+1)$. As $m<(q^{2}+1)(q-1)$, $n_{2}<q$, and we have that
\begin{align}
  m=\frac{n_{2}(q^3+1)-1}{q^2-q-1}=n_{2}(q+1)+\frac{2n_{2}(q+1)-1}{q^2-q-1}.\label{eq:GU4-m-2}
\end{align}
Since $m$ is a integer number, $q^2-q-1$ must divide $2n_{2}(q+1)-1$. Let $u$ be a positive integer number such that $2n_{2}(q+1)-1=u(q^2-q-1)$. Recall that $n_{2}<q$. Then $u(q^2-q-1)=2n_{2}(q+1)-1<2q^2+2q-1$, and so $u\leq 3$. If $u=1$, then $2n_{2}(q+1)-1=q^2-q-1$, and so $q+1$ must divide $q^2-q$, which is impossible. If $u=2$, then $2n_{2}(q+1)-1=2(q^2-q-1)$, and so $q+1$ must divide $2q^2-2q-1=2(q+1)(q-2)+3$, which is impossible. If $u=3$, then $2n_{2}(q+1)-1=3(q^2-q-1)$, and so $q+1$ must divide $3q^2-3q-2=3(q+1)(q-2)+4$. Thus $q+1$ divides $4$, and so $q=3$. In which case $n_{2}=2$, and by \eqref{eq:GU4-m-2}, $m=55/4$, which is impossible.
\end{proof}

\begin{lemma}\label{lem:4}
The subgroup $H_{0}$ cannot be isomorphic to $^{\hat{}}(\SU_{3}(q)\times \SU_{2}(q)):(q+1)$.
\end{lemma}
\begin{proof}
In this case, $v=q^{6}(q^{4}-q^3+q^2-q+1)(q^2+1)$ by  \eqref{eq:v}. It fllows from Lemmas \ref{lem:subdeg} and \ref{lem:six}(e), $k$ must divide $\lambda (q^2-1)(q^{3}+1)$. On the other hand, $k$ divides $\lambda(v-1)=\lambda (q^{2}-q+1)(q^{10}+q^8-q^7+q^4+q^3-q-1)$. Therefore, $k$ divides $\lambda(q^{2}-q+1)\gcd(q^{10}+q^8-q^7+q^4+q^3-q-1,(q-1)(q+1)^2)$. Note that $\gcd(q^{10}+q^8-q^7+q^4+q^3-q-1,(q-1)(q+1)^2)$ divides $9$. Let $m$ be a positive integer that $mk=9\lambda f(q)$, where $f(q)=q^{2}-q+1$. Then by Lemma~\ref{lem:six}(a), we have that
\begin{align}
  k= 1+\frac{m(q^{10}+q^8-q^7+q^4+q^3-q-1)}{9},\label{eq:GU3-k}
\end{align}
where $m<9(q^{2}-q+1)$. Note by~\eqref{eq:k-out} that $k$ divides $2ag(q)$, where $g(q)=q^4(q^3+1)(q^2-1)^2(q+1)$. Then, by~\eqref{eq:GU3-k}, we must have
\begin{align}\label{eq:GU3-q2}
m(q^{10}+q^8-q^7+q^4+q^3-q-1)+9\mid 18ag(q).
\end{align}
Let now $d(q)=q^9-6q^8+4q^7+3q^6+q^5-3q^4-4q^3-2q^2+2q+3$ and $h(q)=q^2+q-3$. Then $18amh(q)[m(q^{10}+q^8-q^7+q^4+q^3-q-1)+9]-18amg(q)=18am[d(q)+9h(q)]$, and so~\eqref{eq:GU3-q2} implies that $m(q^{10}+q^8-q^7+q^4+q^3-q-1)+9$ divides $18am[d(q)+9h(q)]$. Thus $m(q^{10}+q^8-q^7+q^4+q^3-q-1)+9\leq 18am|d(q)+9h(q)|$, and so $q^{10}+q^8-q^7+q^4+q^3-q-1<18a|q^9-6q^8+4q^7+3q^6+q^5-3q^4-4q^3+16q^2+20q-45|$. This inequality holds only for $q \in \{2, 3, 4, 8, 9, 16, 25, 27, 32, 64\}$. For these values of $q$, there is no posible parameters $k$ satisfying~\eqref{eq:GU3-q2}, which is a contradiction.
\end{proof}



\begin{lemma}\label{lem:5}
The subgroup $H_{0}$ cannot be isomorphic to $^{\hat{}}(q+1)^4:\S_{5}$.
\end{lemma}
\begin{proof}
By~\eqref{eq:v}, we have $v=q^{10}(q^5+1)(q^4-1)(q^3+1)(q^2-1)/120\cdot(q+1)^4$, and since $|\Out (X)|= 2a\cdot\gcd(q+1, 5)$, it follows from \eqref{eq:k-out} that $k$ divides $240a(q+1)^4$. By \cite{a:reg-classical,a:Zhou-lam3-classical} and Lemma~\ref{lem:six}(c), we may assume that $\lambda$ is at least $4$, and so
\begin{align*}
 \frac{q^{10}(q^5+1)(q^4-1)(q^3+1)(q^2-1)}{120\cdot(q+1)^4}\leq \lambda v< k^2 \leq 57600a^2(q+1)^8.
\end{align*}
This implies that $q^{10}(q^5+1)(q^4-1)(q^3+1)(q^2-1)<57600 a^2 (q+1)^{12}$. This inequality is true only when $q \in \{2,3\}$. Since $k$ is a divisor of $240a(q+1)^4$, for each such $q=p^a$, the possible values of $k$ and $v$ are listed in Table~\ref{tbl:5}.
\begin{table}
\centering\small
  \caption{Possible value for $k$ and $v$ when $q \in \{2,3\}$.}
  \label{tbl:5}
  \begin{tabular}{cll}
   \noalign{\smallskip}\hline\noalign{\smallskip}
    $q$ & $2$ & $3$ \\
    \hline\noalign{\smallskip}
    $v$ & $1408$ & $19440$  \\
    $k$ divides & $8404641$ & $61440$ \\
   \hline\noalign{\smallskip}
  \end{tabular}
\end{table}
This is a contradiction as for each $k$ and $v$ as in Table~\ref{tbl:5}, the fraction $k(k-1)/(v-1)$ is not integer.
\end{proof}

\begin{lemma}\label{lem:6}
The subgroup $H_{0}$ cannot be isomorphic to $^{\hat{}}\frac{q^5+1}{q+1}:5$.
\end{lemma}
\begin{proof}
In this case~\eqref{eq:v} implies that $v=q^{10}(q^4-1)(q^3+1)(q^2-1)(q+1)/5$, and since $|\Out (X)|= 2a\cdot\gcd(q+1, 5)$, it follows from \eqref{eq:k-out} that $k$ divides $2a(q^4-q^3+q^2-q+1)$. By Lemma \ref{lem:six}(c), we have that
\begin{align*}
q^{10}(q^4-1)(q^3+1)(q^2-1)(q+1)\leq 5\lambda v< 5 k^2 \leq 20a^2 \cdot  (q^4-q^3+q^2-q+1)^2,
\end{align*}
which is impossible.
\end{proof}

\begin{lemma}\label{lem:7}
The subgroup $H_{0}$ cannot be isomorphic to $\,^{\hat{}}\SU_{5}(q_{0})\cdot \gcd(\frac{q+1}{q_{0}+1}, 5)$, where $q=q_{0}^r$ and $r$ is a odd prime number.
\end{lemma}
\begin{proof}
By~\eqref{eq:v}, we have that
\begin{align*}
  v= \frac{1}{b}\cdot \frac{q_{0}^{10r}(q_{0}^{5r}+1)(q_{0}^{4r}-1)(q_{0}^{3r}+1)(q_{0}^{2r}-1)}{q_{0}^{10}(q_{0}^{5}+1)(q_{0}^{4}-1)(q_{0}^{3}+1)(q_{0}^{2}-1)},
\end{align*}
where $b=\gcd(\frac{q+1}{q_{0}+1}, 5)$. Note by \eqref{eq:k-out} that $k$ divides $10aq_{0}^{10}(q_{0}^{5}+1)(q_{0}^{4}-1)(q_{0}^{3}+1)(q_{0}^{2}-1)$. We may assume that $\lambda\geq 4$ by \cite{a:reg-classical,a:Zhou-lam3-classical}. Moreover, $b\in \{1, 5\}$, and $a^2\leq q_{0}^r$ as $q=q_{0}^{r}$. Since $\lambda v<k^2$ by Lemma~\ref{lem:six}(b), we must have
\begin{align*}
 q_{0}^{10r}(q_{0}^{5r}+1)(q_{0}^{4r}-1)(q_{0}^{3r}+1)(q_{0}^{2r}-1)< 100q_{0}^{30+r}(q_{0}^{5}+1)^3(q_{0}^{4}-1)^{3}(q_{0}^{3}+1)^{3}(q_{0}^{2}-1)^{3}.
\end{align*}
Note that $q_{0}^{24r-1}\leq q_{0}^{10r}(q_{0}^{5r}+1)(q_{0}^{4r}-1)(q_{0}^{3r}+1)(q_{0}^{2r}-1)$ and $q_{0}^{30+r}(q_{0}^{5}+1)^3(q_{0}^{4}-1)^{3}(q_{0}^{3}+1)^{3}(q_{0}^{2}-1)^{3}\leq q_{0}^{72+r}$. Then $q_{0}^{23r-1}<100q_{0}^{72}$, and this implies that $r=3$. In which case~\eqref{eq:v} implies that
\begin{align}\label{lem:case8-3-v}
 v=\frac{q_{0}^{20}(q_{0}^{15}+1)(q_{0}^{12}-1)(q_{0}^{9}+1)(q_{0}^{6}-1)}{(q_{0}^{5}+1)(q_{0}^{4}-1)(q_{0}^{3}+1)(q_{0}^{2}-1)\cdot \gcd(q_{0}^{2}-q_{0}+1, 5)}.
\end{align}
 By \eqref{eq:k-out}, $k$ divides $2abq_{0}^{10}(q_{0}^{5}+1)(q_{0}^{4}-1)(q_{0}^{3}+1)(q_{0}^{2}-1)$, where $b=\gcd(q_{0}^{2}-q_{0}+1, 5)$.  Then by Lemma~\ref{lem:six}(c), we have that
$\lambda q_{0}^{20}(q_{0}^{15}+1)(q_{0}^{12}-1)(q_{0}^{9}+1)(q_{0}^{6}-1)< 4a^2b^3q_{0}^{30}(q_{0}^{5}+1)^3(q_{0}^{4}-1)^3(q_{0}^{3}+1)^3(q_{0}^{2}-1)^3$.
Therefore, $\lambda< 4a^2b^3$. Since $k$ divides $2abq_{0}^{10}(q_{0}^{5}+1)(q_{0}^{4}-1)(q_{0}^{3}+1)(q_{0}^{2}-1)$ and $v-1$ is coprime to $q_{0}$, $k$ must divide $2\lambda ab(q_{0}^{5}+1)(q_{0}^{4}-1)(q_{0}^{3}+1)(q_{0}^{2}-1)$. We use again Lemma~\ref{lem:six}(c), and so $\lambda v<k^{2} \leq 4\lambda^2a^2b^2(q_{0}^{5}+1)^2(q_{0}^{4}-1)^2(q_{0}^{3}+1)^2(q_{0}^{2}-1)^2$. Thus \eqref{lem:case8-3-v} implies that
\begin{align}\label{eq:case7-lam}
q_{0}^{47}< \frac{q_{0}^{20}(q_{0}^{15}+1)(q_{0}^{12}-1)(q_{0}^{9}+1)(q_{0}^{6}-1)}{(q_{0}^{5}+1)(q_{0}^{4}-1)(q_{0}^{3}+1)(q_{0}^{2}-1)}< 4\lambda a^2b^3.
\end{align}
Since $\lambda< 4a^{2}b^3$, it follows form \eqref{eq:case7-lam} that $q_{0}^{47}<16a^{4}b^6$, where $b=\gcd(q_{0}^{2}-q_{0}+1, 5)$, which is impossible.
\end{proof}

\begin{lemma}\label{lem:8}
The subgroup $H_{0}$ cannot be isomorphic to $^{\hat{}}{\SO_{5}}(q)$ with $q$ odd.
\end{lemma}
\begin{proof}
In this case, by~\eqref{eq:v}, we have that $v=q^6(q^5+1)(q^3+1)$. It follows from \eqref{eq:k-out} that $k$ divides $2ag(q)$, where $g(q)=q^4(q^4-1)(q^2-1)$. Moreover, Lemma~\ref{lem:six}(a) implies that $k$ divides $\lambda(v-1)$. Let $f(q)=3\cdot(q-1)^2$. Then $\gcd(v-1,2 \cdot q^4(q^4-1)(q^2-1))$ divides $f(q)$, and so $k$ is a divisor of $\lambda a f(q)$. Suppose that $m$ is a positive integer such that $mk=\lambda a f(q)$. Since now $k(k-1)=\lambda(v-1)$, it follows that  $k=1+m(v-1)/af(q)$, and since $k\mid 2ag(q)$, we must have $m(v-1)+a f(q)\mid 2a^{2}f(q)g(q)$. Therefore, $q^6(q^5+1)(q^3+1) <2a^{2}f(q)g(q)$ for $q$ odd, and this does not give rise to any possible parameters.
\end{proof}

\begin{lemma}\label{lem:9-11}
The subgroup $H_{0}$ cannot be isomorphic to the subgroups as in the lines {\rm 9-11} of {\rm Table~\ref{tbl:maxes}}.
\end{lemma}
\begin{proof}
Let $H_{0}$ be isomorphic to one of the subgroups  in the lines {\rm 9-11} of {\rm Table~\ref{tbl:maxes}}. Since $|X|\leq |\Out(X)|^{2}\cdot|H\cap X|^{3}$, we only need to consider the pairs $(X, H\cap X)$  in Table~\ref{tbl:l9-10}. For each such $H\cap X$, by \eqref{eq:v}, we obtain $v$ as in the third column of Table~\ref{tbl:l9-10}. Recall that $k$ is a divisor of $2a\cdot\gcd(5, q+1)\cdot|H\cap X|$ which is recorded in the fourth column of Table~\ref{tbl:l9-10}. This is a contradiction as for each $k$ and $v$ as in Table~\ref{tbl:l9-10}, the fraction $k(k-1)/(v-1)$ is not integer.
\begin{table}
  \centering
  \small
  \caption{The pairs $(X, H\cap X)$ in Lemma \ref{lem:9-11}}\label{tbl:l9-10}
  \begin{tabular}{llll}
  \hline
  \multicolumn{1}{c}{$X$} &
  \multicolumn{1}{l}{$H\cap X$} &
   \multicolumn{1}{l}{$v$} &
   \multicolumn{1}{l}{$k$ divides}\\
  \hline
  $\PSU_5(2)$ & $\PSL_2(11)$ & $20736$ & $1320$\\
  $\PSU_5(4)$ & $^{\hat{}} 5_{+}^{1+2}:\Sp_{2}(5)$ & $3562930176$ & $60000$ \\
  $\PSU_5(9)$ & $^{\hat{}} 5_{+}^{1+2}:\Sp_{2}(5)$ & $1051720694280527616$ & $60000$\\
  \hline
\end{tabular}
\end{table}
\end{proof}

\begin{proof}[\rm \textbf{Proof of Theorem~\ref{thm:main}}]
The proof of the main result follows immediately from Lemmas \ref{lem:1}--\ref{lem:9-11}.
\end{proof}





\begin{thebibliography}{10}
    
    \bibitem{a:ABD-PSL2}
    S.~Alavi, M.~Bayat, and A.~Daneshkhah.
    \newblock Symmetric designs admitting flag-transitive and point-primitive
    automorphism groups associated to two dimensional projective special groups.
    \newblock {\em Designs, Codes and Cryptography}, 79(2): 3337--351, 2016.
    
    \bibitem{a:ABD-PSL3}
    S.~H. Alavi and M.~Bayat.
    \newblock Flag-transitive point-primitive symmetric designs and three
    dimensional projective special linear groups.
    \newblock {\em Bulletin of Iranian Mathematical Society (BIMS)},
    42(1):201--221, 2016.
    
    \bibitem{a:ABD-Exp}
    S.~H. Alavi, M.~Bayat, and A.~Daneshkhah.
    \newblock {Symmetric designs and finite simple exceptional groups of Lie type}.
    \newblock {\em arXiv e-prints}, page arXiv:1702.01257v4, 2019.
    
    \bibitem{a:ABDZ-U4}
    S.~H. Alavi, M.~Bayat, A.~Daneshkhah, and S.~Z. Zarin.
    \newblock Symmetric designs and four dimensional projective special unitary
    groups.
    \newblock {\em Discrete Math.}, 342(4):1159--1169, 2019.
    
    \bibitem{b:beth1999design}
    T.~Beth, D.~Jungnickel, and H.~Lenz.
    \newblock {\em Design Theory:}.
    \newblock Design Theory. Cambridge University Press, 1999.
    
    \bibitem{a:Braic-2500-nopower}
    S.~Brai{\'c}, A.~Golemac, J.~Mandi{\'c}, and T.~Vu{\v{c}}i{\v{c}}i{\'c}.
    \newblock Primitive symmetric designs with up to 2500 points.
    \newblock {\em J. Combin. Des.}, 19(6):463--474, 2011.
    
    \bibitem{b:BHR-Max-Low}
    J.~N. Bray, D.~F. Holt, and C.~M. Roney-Dougal.
    \newblock {\em The maximal subgroups of the low-dimensional finite classical
        groups}, volume 407 of {\em London Mathematical Society Lecture Note Series}.
    \newblock Cambridge University Press, Cambridge, 2013.
    \newblock With a foreword by Martin Liebeck.
    
    \bibitem{a:Camina94}
    A.~R. Camina.
    \newblock A survey of the automorphism groups of block designs.
    \newblock {\em J. Combin. Des.}, 2(2):79--100, 1994.
    
    \bibitem{b:Atlas}
    J.~H. Conway, R.~T. Curtis, S.~P. Norton, R.~A. Parker, and R.~A. Wilson.
    \newblock {\em Atlas of finite groups}.
    \newblock Oxford University Press, Eynsham, 1985.
    \newblock Maximal subgroups and ordinary characters for simple groups, With
    computational assistance from J. G. Thackray.
    
    \bibitem{a:D-PSU3}
    A. Daneshkhah, S. Zang Zarin, 
    \newblock Flag-transitive point-primitive symmetric designs and three dimensional projective special unitary groups, 
    \newblock {\em Bull. Korean
    Math. Soc.}
     54(6): 2029–-2041, 2017.
 
    \bibitem{a:rank3}
    U.~Dempwolff.
    \newblock Primitive rank 3 groups on symmetric designs.
    \newblock {\em Des. Codes Cryptogr.}, 22(2):191--207, 2001.
    
    \bibitem{b:Dixon}
    J.~D. Dixon and B.~Mortimer.
    \newblock {\em Permutation groups}, volume 163 of {\em Graduate Texts in
        Mathematics}.
    \newblock Springer-Verlag, New York, 1996.
    
    \bibitem{GAP4}
    The GAP~Group.
    \newblock {\em {GAP -- Groups, Algorithms, and Programming, Version 4.7.9}},
    2015.
    
    \bibitem{b:Hugh-design}
    D.~Hughes and F.~Piper.
    \newblock {\em Design Theory}.
    \newblock Up (Methuen). Cambridge University Press, 1988.
    
    \bibitem{a:LSS1987}
    M.~W. Liebeck, J.~Saxl, and G.~Seitz.
    \newblock On the overgroups of irreducible subgroups of the finite classical
    groups.
    \newblock {\em Proc. Lond. Math. Soc.}, 50(3):507--537, 1987.
    
    \bibitem{a:reg-reduction}
    E.~O'Reilly-Regueiro.
    \newblock On primitivity and reduction for flag-transitive symmetric designs.
    \newblock {\em J. Combin. Theory Ser. A}, 109(1):135--148, 2005.
    
    \bibitem{a:reg-classical}
    E.~O'Reilly-Regueiro.
    \newblock Biplanes with flag-transitive automorphism groups of almost simple
    type, with classical socle.
    \newblock {\em J. Algebraic Combin.}, 26(4):529--552, 2007.
    
    \bibitem{a:Praeger-45-12-3}
    C.~E. Praeger.
    \newblock The flag-transitive symmetric designs with 45 points, blocks of size
    12, and 3 blocks on every point pair.
    \newblock {\em Des. Codes Cryptogr.}, 44(1-3):115--132, 2007.
    
    \bibitem{a:Praeger-imprimitive}
    C.~E. Praeger and S.~Zhou.
    \newblock Imprimitive flag-transitive symmetric designs.
    \newblock {\em J. Combin. Theory Ser. A}, 113(7):1381--1395, 2006.
    
    \bibitem{a:tits}
    G.~M. Seitz.
    \newblock Flag-transitive subgroups of {C}hevalley groups.
    \newblock {\em Ann. of Math. (2)}, 97:27--56, 1973.
    
    \bibitem{a:Zhou-lam100}
    D.~Tian and S.~Zhou.
    \newblock Flag-transitive point-primitive symmetric {$(v,k,\lambda)$} designs
    with {$\lambda$} at most 100.
    \newblock {\em J. Combin. Des.}, 21(4):127--141, 2013.
    
    \bibitem{a:Zhou-lam3-classical}
    S.~Zhou, H.~Dong, and W.~Fang.
    \newblock Finite classical groups and flag-transitive triplanes.
    \newblock {\em Discrete Math.}, 309(16):5183--5195, 2009.
    
\end{thebibliography}

\end{document}